\documentclass[12pt]{amsart}
\usepackage{amsmath,amsthm,amssymb,euscript,verbatim}
\newtheorem{theorem}{Theorem}
\newtheorem{prop}[theorem]{Proposition}
\newtheorem{lemma}[theorem]{Lemma}

\theoremstyle{definition}

\newcommand{\F}{{F}}

\newcommand{\PP}{\mathbb P}
\newcommand{\R}{\mathbb R}
\newcommand{\C}{\mathbb C}
\newcommand{\CP}{\C\PP}
\newcommand{\CH}{\C{\mathrm H}}
\newcommand{\Lie}{{\mathcal L}}
\def\({\left(}
\def\){\right)}
\def\<{\left <}
\def\>{\right >}
\def\a {\alpha}
\def\b {\beta}
\def\l {\lambda}

\newcommand{\I}{{\mathcal I}}

\newcommand{\w}{\omega}

\newcommand{\bv}{\mathbf v}

\newcommand{\JJ}{{\mathrm J}} 
\newcommand{\RR}{{\sf R}} 

\newcommand{\nat}{\widetilde\nabla}
\newcommand{\RRt}{\widetilde \RR}
\newcommand{\mt}{\widetilde M}
\newcommand{\Mt}{\widetilde M}
\def\intprod{\mathbin{\raisebox{.4ex}{\hbox{\vrule height .5pt width
5pt depth 0pt %
         \vrule height 3pt width .5pt depth 0pt}}}}

\def\&{\wedge}

\begin{document}
\title{The structure Jacobi operator for hypersurfaces in $\CP^2$ and $\CH^2$}
\author{Thomas A. Ivey}
\address{Dept. of Mathematics, College of Charleston\\
66 George St., Charleston SC 29424-0001}
\email{IveyT@cofc.edu}
\author{Patrick J. Ryan}
\address{Department of Mathematics and Statistics\\
              McMaster University\\
              Hamilton, ON  Canada L8S 4K1}
\email{ryanpj@mcmaster.ca}
\keywords{Hopf hypersurface, pseudo-Einstein, exterior differential systems}
\subjclass{53C25, 53C40, 53C55, 58A15}
\date{\today}
\maketitle
\begin{abstract}
Using the methods of moving frames, we study real hypersurfaces in complex projective space
$\CP^2$ and  complex hyperbolic space $\CH^2$ whose
structure Jacobi operator has various special properties.
Our results complement work of several other authors who worked in $\CP^n$ and $\CH^n$ for $n \ge 3$.

\end{abstract}

\section{Introduction}

The complete simply connected K\"ahler manifolds of nonzero
constant holomorphic curvature are the complex space forms $\CP^n$
and $\CH^n$.  Takagi \cite{takagi1975}, for $\CP^n$ and
Montiel \cite{montiel}, for $\CH^n$,
catalogued a specific list of real hypersurfaces
which may be characterized as the homogeneous Hopf
hypersurfaces.  Other
characterizations of these hypersurfaces have been derived over the years, both in terms
of extrinsic information (such as properties of the shape
operator) and intrinsic information (such as properties of the
curvature tensor).  In both cases, the interaction of these
geometric objects with the complex structure has played an
important role.

Occurring as a real hypersurface in $\CP^n$ or $\CH^n$ places
significant restrictions on the geometry of a Riemannian manifold
$M$ and on the way it is immersed.  For example, it is known that
such an $M$ cannot be Einstein (an intrinsic condition) or umbilic (an
extrinsic condition).  In fact, neither the Ricci tensor nor the
shape operator can be parallel.  Nevertheless, elements of the
lists of Takagi and Montiel enjoy many nice properties and
geometers have been successful in characterizing them in terms of
these properties.

Recently,  the {\it structure Jacobi
operator} has been an object of study and various non-existence and
classification results are now known for $n \ge 3$. Unfortunately,
the methods of proof used in establishing these results do not
carry over to the case $n=2$. In this paper, we obtain
corresponding results for  $\CP^2$ and $\CH^2$ using the
method of moving frames, along with the theory of exterior differential systems.

In what follows, all manifolds are assumed connected and all
manifolds and maps are assumed smooth $(C^\infty)$ unless stated otherwise.  Basic
notation and historical information for hypersurfaces in complex space forms
may be found in \cite{nrsurvey}.  For more on
moving frames and exterior differential systems, see the monograph
\cite{BCG3} or the textbook \cite{cfb}.

\subsection{Hypersurfaces in Complex Space Forms}

Throughout this paper, we will take the holomorphic sectional curvature of the complex
space form in question to be $4c$.  The curvature operator $\RRt$ of the
space form satisfies
\begin{equation}\label{ambientcurvature}
\RRt(X,Y) = c (X \wedge Y + JX \wedge JY + 2\langle X, J Y\rangle J)
\end{equation}
for tangent vectors $X$ and $Y$ (cf. Theorem 1.1 in \cite{nrsurvey}),
where $X\wedge Y$ denotes the skew-adjoint operator defined by
$$(X \wedge Y) Z = \langle Y, Z\rangle X - \langle X, Z\rangle Y.$$
We will denote by $r$ the positive number such
that $c = \pm 1/r^2$.  This is the same convention as used in (\cite{nrsurvey}, p. 237).

A real hypersurface $M$ in $\CP^n$ or $\CH^n$ inherits two
structures from the ambient space.  First, given a unit normal
$\xi$, the {\em structure vector field} $W$ on $M$ is defined so
that
$$\JJ W = \xi, \qquad W \in TM,$$
where $\JJ$ is the complex structure.
This gives an orthogonal splitting of the tangent space
$$TM = \operatorname{span} \{W\} \oplus W^\perp.$$
Second, on the tangent space we define a linear operator $\varphi$ which is the complex
structure $\JJ$ followed by projection onto $TM$:
$$\varphi X = \JJ X - \langle X,W\rangle \xi, \qquad \varphi : TM \to TM.$$

Recall that, for a tangent vector field $V$ on
a Riemannian manifold, the {\em Jacobi operator} $R_V$
is a tensor field of type $(1,1)$ satisfying
$$R_V(X) = \RR(X,V)V,$$
where $\RR$ denotes the Riemannian curvature tensor of type $(1,3)$.  Note
that, because of the symmetries of the curvature tensor, $R_V$ is self-adjoint
and $R_V V =0$.
For a real hypersurface in a complex space form in particular,
and $V=W$ (the structure vector),
$R_W$ is called the {\em structure Jacobi operator}.
In this paper, we will characterize certain hypersurfaces in
$\CP^2$ and $\CH^2$ in terms of the structure Jacobi operator.

Some of the results we will state involve the notion of
{\em Hopf hypersurfaces}.
A hypersurface $M$ in a complex space form is said to be a Hopf
hypersurface if the structure vector $W$ is a principal vector,
(i.e. $AW = \alpha W$, where $A$ is the shape operator).   It is a
non-obvious fact (proved by Y. Maeda \cite{YMaeda} for $\CP^n$
and by Ki and Suh \cite{KiSuh} for $\CH^n$) that the principal curvature $\alpha$ is
(locally) constant.  We refer to $\alpha$ as the {\it Hopf principal curvature} following
Martins \cite{martins}.
For an arbitrary oriented hypersurface in a complex space form, we define
the function
$$\alpha = \langle A\,W, W\rangle.
$$
\noindent
Of course, $\alpha$ need not be constant in general.

We also recall the notion of {\em pseudo-Einstein hypersurfaces}.
A real hypersurface $M$ in a complex space form is said to be
pseudo-Einstein if there are constants $\rho$ and $\sigma$ such that
the Ricci (1,1)-tensor $S$ of $M$ satisfies
$$S X = \rho X + \sigma \langle X,W\rangle W$$
for all tangent vectors $X$.

\subsection{Summary of results}

We summarize results of Perez and collaborators on hypersurfaces
satisfying conditions involving the structure Jacobi operator:
\begin{theorem}[Ortega-Perez-Santos \cite{Ortega2006}] \label{OPStheorem}
Let $M^{2n-1}$, where $n \ge 3$, be a real hypersurface in $\CP^n$ or $\CH^n$.
Then the structure Jacobi operator $R_W$ cannot be parallel.
\end{theorem}

\begin{theorem}[Perez-Santos \cite{Perez2005a}]\label{LieParalleltheorem}
Let $M^{2n-1}$, where $n \ge 3$, be a real hypersurface in $\CP^n$.
Then the Lie derivative $\Lie_V R_W$  of the structure Jacobi operator cannot vanish for all
tangent vectors $V$.
\end{theorem}

Weakening the hypothesis of Theorem \ref{LieParalleltheorem}, Perez et al. \cite{Perez2005b}
were able to prove the following.

\begin{theorem}[Perez-Santos-Suh]\label{PSSthm}
Let $M^{2n-1}$, where $n \ge 3$, be a real hypersurface in $\CP^n$.  If $\Lie_W R_W = 0$, then $M$ is a Hopf hypersurface.
If the Hopf principal curvature $\alpha$
is nonzero, then $M$ is locally congruent to a geodesic sphere or a tube over a totally geodesic $\CP^k$, where
$ 0 < k < n-1.$
\end{theorem}

In \S\ref{first} we extend Theorem \ref{OPStheorem} to the case $n=2$, while
at the end of \S\ref{LWRW} we extend Theorem \ref{LieParalleltheorem} to the case $n=2$ for both $\CP^2$ and $\CH^2$.
We find that the analogue of Theorem \ref{PSSthm} for $n=2$ is essentially the same,
and is valid for $\CH^2$ as well as $\CP^2$.  Specifically,
in \S \ref{LWRW}, we prove

\begin{theorem} \label{LWRWtheorem}
Let $M^3$ be a real hypersurface in $\CP^2$ or $\CH^2$.  Then the identity  $\Lie_W R_W =0 $  is satisfied if and
only if $M$ is a pseudo-Einstein hypersurface.
\end{theorem}

It is not immediately obvious that Theorem \ref{LWRWtheorem} is, in fact, the extension of Theorem \ref{PSSthm}
to $\CP^2$ and to $\CH^2$.  The analogue of Theorem \ref{PSSthm} for $n=2$ would say that a hypersurface
$M^3$ in $\CP^2$ with $\Lie_W R_W =0 $ must be an open subset of a geodesic sphere or a Hopf hypersurface
with $\alpha = 0$.  However, the classification of pseudo-Einstein hypersurfaces in $\CP^2$ by Kim and Ryan \cite{kimryan2}
yields exactly the same list of hypersurfaces.  The classification of pseudo-Einstein hypersurfaces
in $\CH^2$ by Ivey and Ryan \cite{IveyRyan} yields an analogous list -- open subsets of horospheres, geodesic spheres,
tubes over $\CH^1$, and Hopf hypersurfaces with $\alpha = 0$.

It is not hard to check that every Hopf hypersurface
with $\alpha = 0$ (in $\CH^n$ as well as in $\CP^n$)
for $n \ge 2$, satisfies $\Lie_W R_W = 0$.
The structure theory for Hopf hypersurfaces
with $\alpha = 0$ is described in \cite{cecilryan, IveyRyan, kimryan2, martins}.
Note that such hypersurfaces need not be pseudo-Einstein when $n \ge 3$.  On the other hand, there are some
pseudo-Einstein hypersurfaces in $\CP^n$, where $n \ge 3$, that do not satisfy
$\Lie_W R_W =0 $.  Thus one cannot restate Theorem \ref{PSSthm}
in terms of the pseudo-Einstein condition.

Finally, we observe that the condition considered in Theorem \ref{LieParalleltheorem} is actually quite strong.  In
\S \ref{LieParallelSection} we provide a new proof of this theorem that is also valid for $\CH^n$.

\section{Basic Equations}
\label{basic}

In this and subsequent sections, we follow the notation and terminology of \cite{nrsurvey}:
$M^{2n-1}$ will be a hypersurface
in a complex space form $\mt$ (either $\CP^n$ or $\CH^n$) having constant holomorphic sectional curvature $4c\ne 0$.
The structures $\xi$, $W$, and $\varphi$ are
as defined in the Introduction. The $(2n-2)$-dimensional
distribution $W^\perp$ is called the {\it holomorphic distribution}. The operator $\varphi$ annihilates
$W$ and acts as complex structure on $W^\perp$. The shape
operator $A$ is defined by
$$A X = -\nat_X \xi$$
where $\nat$ is the Levi-Civita connection of the ambient space.
The Gauss equation expresses the curvature operator of $M$ in terms
of $A$ and $\varphi$, as follows:
\begin{equation}\label{gausseq}
\RR(X,Y)=AX\& AY+c\(X\& Y+\varphi X\& \varphi Y +2 \<X,\varphi Y\>
\varphi\).
\end{equation}
\noindent
In addition, it is easy to show (see \cite{nrsurvey}, p. 239) that
\begin{equation} \label{nablaW}
\nabla_X W=\varphi A X,
\end{equation}
where $\nabla$ is the Levi-Civita connection of the hypersurface $M$.

Consider now the case $n=2$, so that $M^3$ is a hypersurface in $\CP^2$ or $\CH^2$.
Suppose that there is a point $p$ (and hence an open neighborhood of $p$) where $AW \ne \alpha W$.
Then there is a positive function $\beta$ and a unit vector field $X \in W^\perp$ such that
$$AW = \alpha W + \beta X.$$
Let $Y = \varphi X$.  Then there are smooth functions $\lambda$, $\mu$, and $\nu$ defined near $p$ such
that with respect to the orthonormal frame $(W,X,Y)$,
\begin{equation} A =
 \begin {pmatrix} \a&\b& 0 \\
                \b &\l & \mu \\
                            0&\mu&\nu \\
\end {pmatrix}\label{shapematrix}.
\end{equation}
A routine computation, using the Gauss equation \eqref{gausseq}, yields
\begin{equation} R_W =
 \begin {pmatrix}\label{jacobimatrix} 0&0& 0 \\
                0 &\a\l+c-\b^2 & \a\mu \\
                            0&\a\mu&\a\nu+c \\
\end {pmatrix}.
\end{equation}

Consider now a point where $AW = \a W$.  Let $X$ be a unit principal vector in $W^\perp$ and let $Y = \varphi X$.
Then there are numbers $\a$, $\l$ and $\nu$ such that equations \eqref{shapematrix} and \eqref{jacobimatrix}
still hold at this point, but with $\b = \mu =0$.

In this connection, we recall the following useful fact (\cite{nrsurvey}, p. 246.)
\begin{prop}\label{hopfpc} Let $M^{2n-1}$, where $n \ge 2$, be a Hopf hypersurface in $\CP^n$ or $\CH^n$
with Hopf principal curvature $\alpha$.
If $X$ is a unit vector in $W^\perp$ such that $AX = \l X$ and $A \varphi X = \nu \varphi X$, then
\begin{equation}
\l \nu =\frac{\l + \nu}{2}\ \alpha +c.
\end{equation}
\end{prop}

\section{Parallelism of $R_W$}\label{first}

\subsection{The condition $\nabla R_W = 0$.}\label{parallel}
We first show that this condition implies $R_W = 0$.
\begin{prop}\label{parallelprop}
 Let $M^{2n-1}$ be a hypersurface in $\CP^n$ or $\CH^n$, where $n \ge 2$.
 If $\nabla R_W = 0$ on $M$, then $R_W = 0$.
\end{prop}
\begin{proof}
Since $R_W$ is parallel, every curvature operator commutes with $R_W$.  Then for any tangent vector $V$,
$$
0 = \RR(V, W) R_W W = R_W \RR(V, W) W = R_W^2 V,$$
and thus $R_W^2=0$.  So $R_W$, being self-adjoint, must also vanish.
\end{proof}

\subsection{The condition $R_W = 0$}
\label{RWzero}

\begin{prop} \label{nonvanish} There are no hypersurfaces in $\CP^2$ or $\CH^2$ such that the structure Jacobi operator
$R_W$ vanishes identically.
\end{prop}
\begin{proof} We use the setup from \S \ref{basic} with $n=2$.
First look at possibility of a Hopf hypersurface with $R_W$ = 0.  We see from \eqref{jacobimatrix}
with $\beta=0$ that $\a \l +c =\a \nu + c = 0$, so that $\a \ne 0$ and $\nu = \l \ne 0$.  However, in view of
Proposition \ref{hopfpc},  we have $ 0 =  \a \l + c = \l^2$,
which is a contradiction.

The non-Hopf case is handled by the following proposition which follows directly from \eqref{shapematrix}
and \eqref{jacobimatrix}.  Then Lemma \ref{parallellemma} completes our proof.
\end{proof}

\begin{prop}  \label{parallelconditions}
Suppose that $M^3$ is a non-Hopf hypersurface in $\CP^2$ or $\CH^2$ satisfying $R_W=0.$
Then, in a neighborhood of some point $p$, we have
(using the basic setup of \S \ref{basic})
\begin{itemize}
\item $\b$ and $\alpha$ are nonzero;
\item $\mu = 0$;
\item $\b^2 = \a\l+c$;
\item $\a\nu + c =0.$
\end{itemize}
Conversely, every hypersurface satisfying these conditions will have $R_W = 0.$
\end{prop}
\begin{lemma}\label{parallellemma}
There does not exist a hypersurface in $\CP^2$ or $\CH^2$ satisfying
the conditions of Proposition \ref{parallelconditions}.
\end{lemma}

We prove this lemma in \S \ref{movingframes} using exterior differential systems.

\section{Lie Parallelism of $R_W$}\label{LWRW}
We begin by deriving a necessary condition for a hypersurface to satisfy $\Lie_W R_W = 0$.
\begin{prop}\label{necessary}
For any hypersurface in $\CP^n$ or $\CH^n$, where $n \ge 2$, satisfying $\Lie_W R_W = 0$, we must have
\begin{equation}
[R_W, [\varphi, A]] = 0.
\end{equation}
\end{prop}
\begin{proof}

\begin{align*}\label{uglyv}
(\Lie_W R_W)V &= \Lie_W (R_W V) - R_W(\Lie_W V)\\
&= \nabla_W(R_W V) - \nabla_{R_W V} W -R_W (\nabla_W V) + R_W (\nabla_V W)\\
&= (\nabla_W R_W) V - \varphi A (R_W V) + R_W (\varphi A V)
\end{align*}
for all tangent vectors $V$.
(Here we have used (\ref{nablaW})).
Thus $\Lie_W R_W = 0$ if and only if $\nabla_W R_W = -[R_W, \varphi A]$.
Using the fact that $\nabla_W R_W$ is self-adjoint, we see that $\Lie_W R_W = 0$ implies that
$$(R_W \varphi A - \varphi A R_W)^t = (R_W \varphi A - \varphi A R_W),$$
which, once we use the fact that $A, R_W$ are self-adjoint while $\varphi$ is skew-adjoint, reduces to the desired identity.
\end{proof}
\subsection{The non-Hopf case}\label{LWRWnonHopf}

\begin{prop} \label{LWRWprop}
Suppose that $M^3$ is a non-Hopf hypersurface in $\CP^2$ or $\CH^2$ satisfying $\Lie_WR_W=0.$
Then, in a neighborhood of some point $p$, we have
(using the basic setup of \S \ref{basic})
\begin{itemize}
\item $\b$ and $\alpha$ are nonzero;
\item $\mu = 0$;
\item $\l = \nu$;
\item $\a\nu + c =0$
\end{itemize}
\end{prop}

\begin{proof}
By Proposition \ref{necessary}, we get $R_W(\varphi A - A \varphi)W = 0$ which
implies that $R_W \varphi A W = 0$.  In the setup of \S \ref{basic} with $\b > 0$,
this gives $R_W Y = 0.$  From equation \eqref{jacobimatrix},
we get $\alpha \mu=0$ and $\alpha \nu + c = 0$; the latter guarantees
that $\alpha \ne 0$, and hence $\mu = 0$.  Following the same procedure with $X$, we get
$R_W (\varphi A - A \varphi)X = R_W (\l - \nu)Y = 0$.  Therefore, $(\varphi A - A \varphi)R_W X = 0$,
which reduces to $(\a \l + c - \b^2)(\l - \nu)= 0$.  If $\l \ne \nu$ at some point, then $\a \l + c - \b^2$ vanishes
in a neighborhood of this point and $R_W = 0$ there.  This contradicts Proposition \ref{nonvanish} so we must
conclude that $\l = \nu$ and that in a neighborhood of $p$, we have

\begin{equation} A =
 \begin {pmatrix} \a&\b& 0 \\
                \b &-\frac{c}{\a} & 0 \\
                            0&0&-\frac{c}{\a} \\
\end {pmatrix}\label{LWRWshapematrix}
\end{equation}
and
\begin{equation} R_W =
 \begin {pmatrix}\label{LWRWjacobimatrix} 0&0& 0 \\
                0 &-\b^2 & 0 \\
                            0&0&0 \\
\end {pmatrix}.
\end{equation}
\end{proof}

However, the situation described in Proposition \ref{LWRWprop} cannot, in fact, occur.

\begin{lemma}\label{LWRWlemma}
There does not exist a hypersurface in $\CP^2$ or $\CH^2$ satisfying the conditions listed
in Proposition \ref{LWRWprop}.
\end{lemma}
We prove this in \S \ref{movingframes} using exterior differential systems.  Thus, we have,
\begin{prop} \label{LWRWmustbeHopf}
Let $M^3$ be a real hypersurface in $\CP^2$ or $\CH^2$ such that $\Lie_W R_W = 0.$  Then $M$ must be a Hopf
hypersurface.
\end{prop}
We classify such hypersurfaces in the next section.

\subsection{The Hopf case}\label{LWRWHopf}
Now consider a Hopf hypersurface $M^3$ in $\CP^2$ or $\CH^2$.  At any point of $M$, let $X$ be a unit
principal vector in $W^\perp$ and let $Y = \varphi X$.  Then, with respect to the frame $(W, X, Y)$, we
have
\begin{equation} A =
 \begin {pmatrix} \a&0& 0 \\
                0&\l & 0 \\
                            0&0&\nu \\
\end {pmatrix}\label{Hopfshapematrix}
\end{equation}
and
\begin{equation} R_W =
 \begin {pmatrix}\label{Hopfjacobimatrix} 0&0& 0 \\
                0 &\a\l+c & 0 \\
                            0&0&\a\nu+c \\
\end {pmatrix}.
\end{equation}
\noindent
By a straightforward calculation, we obtain
\begin{equation}
\begin{aligned}\label{LWRWHopfcondition}
[R_W, [\varphi, A]] W &= 0,\\
[R_W, [\varphi, A]] X &= -\a(\l - \nu)^2 Y,\\
[R_W, [\varphi, A]] Y &= \a(\l - \nu)^2 X.
\end{aligned}
\end{equation}
We are now ready to prove the following proposition.
\begin{prop} \label{LWRWHopfprop}
Let $M^3$ be a Hopf hypersurface in $\CP^2$ or $\CH^2$.  Then the identity  $\Lie_W R_W =0 $  is satisfied if and
only if at each point of $M$, one of the following holds:
\begin{itemize}\label{pseudoE}
\item $\a = 0$, $\l \ne \nu$ and $\l \nu = c$;
\item $\a = 0$, $\l = \nu$, and $\l^2 = c$;
\item $\a^2+4c = 0$ and $\l=\nu = \frac{\a}{2}$; or
\item $\a\ne 0$, $\a^2 +4c >0$, $\l = \nu$, and $\l^2 = \a\l + c$.
\end{itemize}.
\end{prop}
\begin{proof}
The necessity of these conditions follows immediately from (\ref{LWRWHopfcondition}),
Proposition \ref{necessary} and Proposition \ref{hopfpc}.
Now suppose that these conditions are satisfied.

If $\a = 0$, we see that $R_W V = c V$ for all $V \in W^\perp$.
If $\a^2 + 4c = 0$, then $\lambda=\nu=\a/2$ and $\a\l+c = -c$, so that $R_W V = -c V$ for all $V \in W^\perp$.  In the remaining case,
$R_W V = \l^2 V$ for all $V \in W^\perp$.  In each case, there is a nonzero constant $k$ such that the
identity $R_W V = k V$ holds globally for all $V \in W^\perp$.  Then for any vector field $V \in W^\perp$,
we have, using (\ref{nablaW})
\begin{equation}
\begin{aligned}
(\Lie_W R_W)V &=  \Lie_W (k V) - k \left(\Lie_W V - \<\Lie_W V, W\> W\right) \\
&= k( \<\nabla_W V, W \> -\<\nabla_V W, W\>) W  \\
&= -k \< V, \nabla_W W \> W =  -k \<V, \varphi A W\> W = 0.
\end{aligned}
\end{equation}
Since $(\Lie_W R_W)W=0$ automatically, we have $\Lie_W R_W = 0$ as required.
\end{proof}

Note that according to Propositions 2.13 and 2.21 of \cite{kimryan2},
the conditions in Proposition \ref{LWRWHopfprop} are precisely the conditions for $M$ to be a pseudo-Einstein hypersurface.
Thus we have completed the proof
of Theorem \ref{LWRWtheorem}.

With a little additional work, we can now prove the analogue of Theorem \ref{LieParalleltheorem} for $n=2$.
\begin{theorem} \label{LieParallel2}
Let $M^3$ be a real hypersurface in $\CP^2$ or $\CH^2$.
Then the Lie derivative $\Lie_V R_W$  of the structure Jacobi operator cannot vanish for all
tangent vectors $V$.
\end{theorem}
\begin{proof}
We suppose that $\Lie_V R_W$ vanishes for all $V$ and derive a contradiction.  First note that we must have
$\Lie_W R_W = 0$.  By Proposition \ref{LWRWmustbeHopf}, M must be Hopf.  Thus, the classification of
Proposition \ref{LWRWHopfprop} can be applied.  For any unit
vector field $V \in W^{\perp}$, and $U = \varphi V$, consider
\begin{equation}
\begin{aligned}
\<(\Lie_V R_W)U, W\> &=  \<(\Lie_V (R_W U) - R_W(\Lie_V U)), W\>\\
&= \<(k \Lie_V U -  k (\Lie_V U - \<\Lie_V U, W\> W)), W\>\\
&= k \< \Lie_V U, W \>\\
&= k \<\nabla_V U, W \> - k \<\nabla_U V, W\> \\
&= -k \< U, \varphi A V\>  +k \<V, \varphi A U\>.
\end{aligned}
\end{equation}
Now fix a particular point $p$, and suppose that $V$ is principal at $p$, with $AV = \l V$.
Then $U$ must also be principal at $p$.  Writing
$AU = \nu U$, we get
$$
\<(\Lie_V R_W)U, W\> = -k (\l + \nu).
$$
at $p$.
The right side of this equation is nonzero unless $\l = - \nu$.  Except possibly for the first case ($\alpha=0$,
$\l \ne \nu$, $\l\nu = c$) in Proposition \ref{LWRWHopfprop}, we have an immediate contradiction.
In the remaining case, our argument shows that the principal curvatures sum to zero everywhere (since $p$
was arbitrary).  However, $\l = -\nu$
locally would give $\l^2 = -c$ and force $\l$ and $\nu$ to be locally constant.
Since the well-known list of Hopf
hypersurfaces with constant principal curvatures does not admit this possibility
(see Theorem 4.13 of \cite{nrsurvey}), our proof is complete.
\end{proof}
%
\section{Lie parallelism for $n \ge 3$}\label{LieParallelSection}
The condition of ``Lie parallelism" (see \cite {Perez2005a}, p. 270) is very strong.
In fact, a tensor field of type $(1,1)$
will be Lie parallel if and only if it is a constant multiple of the identity.
\begin{lemma}\label{LieParallellemma}
Let $T$ be a tensor field of type $(1,1)$ on a manifold $M^n$, where $n \ge 2$.
Then the Lie derivative $\Lie_X T$  vanishes for all vector fields $X$ if and only if $T$ is
a constant multiple of the identity.
\end{lemma}
\begin{proof}
Let $X$ and $Y$ be vector fields and $f$ a real-valued function defined on an open set $U \subset M$.
Then, it is easy to check that the identity
\begin{equation}
(\Lie_{f X} T) Y = f (\Lie_X T) Y - df(TY) X + df(Y) T X
\end{equation}
holds.

Suppose now that $\Lie_V T = 0$ for all vector fields $V$.
Then
\begin{equation}
df(T Y) X =  df(Y) T X
\end{equation}
for all $X$, $Y$, $f$.  For a suitable choice of $Y$ and $f$, we can assume that $df(Y)$ is nonvanishing on $U$,
so that we can write $TX = \tau X$ for a function $\tau= df(TY)/df(Y)$.  Since $\tau$ can depend only on $X$,
it must be independent of $X$ and $Y$.  Therefore, there is a real-valued function $\tau$ such that $T = \tau I$.
Finally, for any vector field $V$,
we have
$$0 = (\Lie_V T)Y = \Lie_V (\tau Y) - \tau \Lie_V Y = d\tau(V) Y
$$
so that $\tau$ must be locally constant.  Conversely, the same equation shows that if $T$ is a constant multiple of the
identity, then $\Lie_V T = 0$.
\end{proof}

We are now in a position to prove our theorem.

\begin{theorem}\label{NewLieParalleltheorem}
Let $M^{2n-1}$, where $n \ge 3$, be a real hypersurface in $\CP^n$ or $\CH^n$.
Then the Lie derivative $\Lie_V R_W$  of the structure Jacobi operator cannot vanish for all
tangent vectors $V$.
\end{theorem}

\begin{proof}
  Suppose that $\Lie_V R_W = 0$ for all $V$.
  Applying the preceding lemma to the $(1,1)$ tensor field $R_W$, we get that $R_W$ is a constant
  multiple of the identity.  Since $R_W W = 0$, we have, in fact, that $R_W = 0$.  Our result is now
  immediate from Theorem \ref{OPStheorem}.
\end{proof}

We could proceed similarly in the $n=2$ case, invoking Proposition \ref{parallelprop}.  This would
provide an alternative proof of Theorem \ref{LieParallel2}.

\section{Differential Forms Calculations}\label{movingframes}

In this section, we prove Lemmas \ref{parallellemma} and \ref{LWRWlemma}
by analyzing the conditions that a moving frame along the hypersurface
would have to satisfy, as a section of the orthonormal frame bundle of the relevant
complex space form $\Mt = \CP^2$ or $\CH^2$.  The conditions proposed in the lemmas
will imply that the sections are integral submanifolds of certain exterior differential
systems on the frame bundle.  The generators of these systems are defined in terms
of the natural coframing on the frame bundle, which we will briefly review.

On the orthonormal frame bundle $\F_o$ of a $n$-dimensional Riemannian manifold $\Mt$,
we define the {\em canonical 1-forms} $\w^i$ and the {\em connection 1-forms}
$\w^i_j$ (where $1\le i,j,k \le n$) by the following properties:
if $(e_1,\ldots,e_n)$ is any orthonormal frame defined on an open set $U\subset \Mt$, and $f:U \to \F_o$
is the corresponding local section, then
\begin{align}
\bv &= (\bv \intprod f^* \w^k) e_k, \label{omegacanon}\\
\nat_\bv e_j &= (\bv \intprod f^* \w^k_j ) e_k, \label{omegaconn}
\end{align}
for any tangent vector $\bv$ at a point in $U$, where $\nat$ denotes
the Levi-Civita connection on $\Mt$ and we use the summation convention.
The connection forms satisfy $\w^j_i=-\w^i_j$.  The forms $\w^i$ and $\w^i_j$ (for $i>j$)
together form a basis for the cotangent space of $\F_o$ at each point.
They satisfy the structure
equations
\begin{align}d\w^i &= -\w^i_j \& \w^j,\\
d\w^i_j &= -\w^i_k \& \w^k_j + \Phi^i_j,
\end{align}
where the 2-forms $\Phi^i_j$ pull back along any section to
give the components of the curvature tensor with respect to the corresponding
frame, i.e., $f^* \Phi^i_j(e_k,e_\ell) = \langle e_i, \RRt(e_k, e_\ell) e_j\rangle$.

In our case, $n=4$ and $\Mt$ is a complex space form.
We will use moving frames that are adapted to the complex structure on $\Mt$ in the following
way:
$$e_4 = \JJ e_1, \qquad e_3 = \JJ e_2.$$
We will refer to these as {\em unitary frames}, and let $\F_u \subset \F_o$ be the
sub-bundle of such frames.  We restrict the canonical and connection forms to $\F_u$ without
change of notation.  The structure group of this sub-bundle is
the 4-dimensional group $U(2)\subset SO(4)$.  Because $\JJ$ is parallel,
only the connection forms $\w^3_2$, $\w^4_1$, $\w^4_2$, $\w^4_3$ are linearly independent,
the remaining forms satisfying the relations
$$\w^2_1 = -\w^4_3, \qquad \w^3_1 = \w^4_2.$$
Using \eqref{ambientcurvature} and the structure equations, we find that the
curvature forms on $\F_u$ satisfy
\begin{align*}
\Phi^3_2 &= c (4 \w^3 \& \w^2 + 2 \w^4 \& \w^1), \\
\Phi^4_1 &= c( 4 \w^4 \& \w^1 +2\w^3 \& \w^2),\\
\Phi^4_2 &= \Phi^3_1 = c(\w^3 \& \w^1 + \w^4 \& \w^2), \\
\Phi^4_3 &= \Phi^1_2=c (\w^1 \& \w^2 + \w^4 \& \w^3). \\
\end{align*}

Along a real hypersurface $M \subset \Mt$, we will use an {\em adapted} moving frame,
meaning a unitary frame such that $e_4$ is normal to the hypersurface (and thus
$e_1$ is the structure vector).  It follows from \eqref{omegacanon} that $f^*\w^4 = 0$
and $f^*(\w^1 \& \w^2 \& \w^3)$ is a nonzero 3-form at each point.  It also follows from
\eqref{omegaconn} that
$$\quad f^*\w^4_i = h_{ij} f^* \w^j, \qquad 1 \le i,j \le 3,$$
where $h_{ij}$ are functions that give the components of the shape operator of $M$.
In particular, working in a neighborhood of a point where $AW \ne \alpha W$, let $W,X,Y$ be the unit
vector fields defined in \S\ref{basic}.  Then $e_1=W, e_2=X, e_3=Y$ and $e_4=\xi$
give the components of an adapted framing, and the $h_{ij}$ are the entries
of the matrix given by \eqref{shapematrix}.

We now have all the tools necessary to prove the two lemmas.
\begin{proof}[Proof of Lemma \ref{parallellemma}]
Again, let $W,X,Y$ be unit vector fields on an open set $U\subset M$, as in \S\ref{basic},
and let $f$ be the adapted moving frame such that $e_1=W, e_2=X, e_3=Y$.
Then
$f$ immerses $U$ as a three-dimensional submanifold of $\F_u$ on which $\w^4=0$
and the $\w^4_i$ satisfy
\begin{equation}\label{greeksetup}
\begin{aligned}
\w^4_1 &= \alpha \w^1 + \beta \w^2,\\
\w^4_2 &= \beta \w^1 + \lambda \w^2 + \mu \w^3,\\
\w^4_3 &=  \mu \w^2 + \nu \w^3,
\end{aligned}
\end{equation}
for some functions $\alpha,\beta,\lambda,\mu,\nu$ satisfying the conditions in the lemma.
Because we assume that $\alpha$ is nowhere vanishing, these conditions can be expressed as
$$\alpha, \beta \ne 0, \qquad \lambda = \dfrac{\beta^2-c}{\alpha}, \qquad \mu = 0, \qquad \nu=-\dfrac{c}{\alpha}.$$
Under these conditions, the functions $\alpha$ and $\beta$ completely determine the second fundamental
form (and hence, determine the hypersurface up to rigid motion).  The proof will proceed by deriving an overdetermined
system of differential equations that these functions must satisfy, and showing that no solutions exist satisfying
the nonvanishing conditions.

Take $(\alpha, \beta)$ as coordinates on $\R^2$, and let $\Sigma \subset \R^2$
be the subset where $\alpha \ne 0$ and $\beta\ne 0$.  On $\F_u \times \Sigma$
define the 1-forms
$$
\begin{aligned}
\theta_0 &= \w^4, \\
\theta_1 &= \w^4_1 - \alpha \w^1 - \beta \w^2,\\
\theta_2 &=\w^4_2 - \beta \w^1 - \dfrac{(\beta^2-c)}{\alpha}\ \w^2, \\
\theta_3 &=\w^4_3 + \dfrac{c}{\alpha}\ \w^3.
\end{aligned}
$$
Then for any adapted frame $f$ along $M$, the image of the
map $p \mapsto (f(p),\alpha(p),\beta(p))$ is a 3-dimensional submanifold in $\F_u\times \Sigma$
which is an {\em integral}
of the Pfaffian exterior differential
system generated by $\theta_0, \theta_1, \theta_2, \theta_3$.  In other words, all 1-forms in this
span pull back to be zero on this submanifold.  We will now investigate the
set of such submanifolds, satisfying the independence condition $\w^1 \& \w^2 \& \w^3 \ne 0$, which is
implied by (\ref{omegacanon}).

Along any such submanifold, the exterior derivatives of the $\theta_i$ must also vanish (i.e.,
they pull back to the submanifold to be zero).
Therefore, we will obtain additional differential forms that must vanish along integral manifolds
if we compute the derivatives of the 1-form generators modulo the algebraic ideal (under wedge product)
generated by those 1-forms.  In this case,
we compute $d\theta_0 \equiv 0$ and
\begin{equation}\label{I1struct}
\begin{aligned}
-d\theta_1 &\equiv \pi_1 \& \w^1 + \pi_2 \& \w^2 + \pi_3 \& \w^3,\\
-d\theta_2  &\equiv \pi_2 \& \w^1 +
\left(\dfrac{2 \beta}{\alpha}\ \pi_2 - \dfrac{(\beta^2-c)}{\alpha^2}\ \pi_1\right) \& \w^2
+\dfrac{\beta}{\alpha}\ \pi_3 \& \w^3, \\
-d\theta_3 &\equiv \pi_3 \& \left(\w^1 +\dfrac{\beta}{\alpha}\w^2\right)
+ \dfrac{c}{\alpha^2}\ \pi_1 \& \w^3
\end{aligned}
\mod \theta_0,\theta_1,\theta_2,\theta_3,
\end{equation}
where
\begin{align*}
\pi_1 &:= d\alpha + 3\dfrac{\beta(\alpha^2-c)}{\alpha}\ \w^3,\\
\pi_2 &:= d\beta + \dfrac{(3\alpha^2\beta^2 + c^2 -c\beta^2)}{\alpha^2}\ \w^3,\\
\pi_3 &:=\beta \w^3_2 +4\alpha\beta \w^1 +
\dfrac{(4\alpha^2\beta^2-c^2+c\beta^2)}{\alpha^2}\ \w^2.
\end{align*}
On any integral submanifold satisfying the independence condition, $\pi_1, \pi_2, \pi_3$
must restrict to be linear combinations of $\w^1, \w^2, \w^3$ at each point.  The
possibilities for these linear combinations are determined by the requirement that
the right-hand sides in \eqref{I1struct} must be zero.  In fact, there is only one parameter's
worth of possible values for the $\pi$'s, given by
\begin{equation}\label{I1pival}
\begin{aligned}
\pi_1 &= \rho(\alpha \w^1 + \beta \w^2),\\
\pi_2 &= \rho\left(\beta \w^1 + \dfrac{\beta^2+c}{\alpha}\w^2\right),\\
\pi_3 &= \dfrac{\rho c}{\alpha}\ \w^3
\end{aligned}
\end{equation}
in terms of the single parameter $\rho$.
In other words, along each submanifold there will be a function $\rho$ such
that the above equations hold.  (To see why, note that the vanishing of the third line of \eqref{I1struct} implies that
$\pi_1,\pi_3$ must be linear combinations of $\w^3$ and $\alpha \w^1 + \beta\w^2$.
On the other hand, linearly combining the first two lines to eliminate the $\pi_3 \& \w^3$
term reveals that $\pi_1,\pi_2$ must be linear combinations of $\w^1, \w^2$.  Thus,
$\pi_1$ must be a multiple of $\alpha \w^1 + \beta\w^2$.  By substituting
this into the right-hand sides of \eqref{I1struct}, we see that this multiple determines
the values of $\pi_2$ and $\pi_3$ at any point.)

Just as we did with $\alpha$ and $\beta$,
we introduce $\rho$ as a new coordinate, and define the following
1-forms on $\F_u \times \Sigma \times \R$:
\begin{align*}
\theta_4 &= \pi_1 - \rho(\alpha \w^1 + \beta \w^2),\\
\theta_5 &= \pi_2 -\rho\left(\beta \w^1 + \dfrac{\beta^2+c}{\alpha}\w^2\right),\\
\theta_6 &= \pi_3 - \dfrac{\rho c}{\alpha}\ \w^3.
\end{align*}
Then for any adapted framing $f$ along $M$ satisfying our assumptions, the image of the map
$p \mapsto (f(p),\alpha(p),\beta(p),\rho(p))$ is an integral submanifold
of the Pfaffian system defined by the 1-forms $\theta_0, \ldots, \theta_6$.
(In technical terms, this system is the {\em prolongation} of the previous one.)

As before, we compute the exterior derivatives of these 1-forms modulo themselves.
We find that
$$d\theta_4 \& \left(\alpha \w^1 + \beta \w^2\right)\equiv
\dfrac{8 c (\alpha^2-c) \rho}{\alpha}\ \w^1 \& \w^2 \& \w^3$$
modulo $\theta_0, \ldots, \theta_6$,
indicating that any integral submanifold satisfying the independence
condition must have $\rho(\alpha^2 - c)=0$ at each point.  (Recall that
the ambient curvature $c$ is nonzero.)
If $\rho \ne 0$ at a point on the submanifold, then $\alpha^2 =c$ on an open set about that point.
However, we compute
$$d\left(\alpha \theta_5  - \beta \theta_4 \right) \& \w^2 \equiv
 \dfrac{2 c(\beta^2 - 2(\alpha^2-c))\rho}{\alpha}\ \w^1 \& \w^2 \& \w^3,$$
which shows that $\beta$ must vanish on that open set, a contradiction.
Therefore, we conclude that $\rho$ must be identically zero on any integral satisfying
the independence condition.
We restrict the system to the submanifold where $\rho=0$.  Then we compute
$$
\begin{aligned}
d\left(\alpha \theta_5  - \beta \theta_4\right)
&\equiv \dfrac{c(2\beta^2+c)(4\alpha^2+\beta^2-c)}{\alpha^2}\ \w^1 \& \w^2,\\
d\theta_6 \& \w^2 &\equiv\dfrac{c(10\alpha^2 \beta^2 - c(4\alpha^2+\beta^2-c))}{\alpha^3}
\ \w^1 \& \w^2 \& \w^3.
\end{aligned}
$$
The first line can vanish only if
$4\alpha^2 + \beta^2=c$, whereupon the vanishing of the last line implies that
one of $\alpha$ or $\beta$ must be zero, a contradiction.  Thus, no hypersurfaces
exist satisfying the hypotheses of the lemma.
\end{proof}

\begin{proof}[Proof of Lemma \ref{LWRWlemma}]
Again, let $W,X,Y$ be unit vector fields on an open set $U\subset M$, satisfying
the conditions given in \S\ref{basic},
and let $f$ be the adapted moving frame such that $e_1=W, e_2=X, e_3=Y$.
Then $f$ immerses $U$ as a 3-dimensional submanifold of $ F_u$.  We note that $f^* \w^4 = 0$
and \eqref{greeksetup} hold for functions $\alpha, \beta,\lambda,\mu,\nu$
satisfying the conditions in the lemma, which can be expressed as
$$\beta,\lambda \ne 0, \qquad \alpha = -c/\lambda,  \qquad \nu=\lambda, \qquad \mu=0.$$

Thus, we set up an exterior differential system $\I$ on $\F_u \times \Sigma$ (where
now $\beta,\lambda$ are the nonzero coordinates on the second factor) generated
by 1-forms
$$
\begin{aligned}
\theta_0 &= \w^4, \\
\theta_1 &= \w^4_1 + (c/\lambda) \w^1 - \beta \w^2,\\
\theta_2 &=\w^4_2 - \beta \w^1 - \lambda \w^2, \\
\theta_3 &=\w^4_3 -\lambda \w^3.
\end{aligned}
$$
Then for any adapted frame $f$ along $M$, the
image of the map $p \mapsto (f(p),\beta(p),\lambda(p))$
will be a 3-dimensional integral submanifold of $\I$ satisfying
the usual independence condition.

We compute $d\theta_0 \equiv 0$ and
\begin{equation}\label{I2struct}
\begin{aligned}
-d\theta_1 &\equiv \dfrac{c}{\lambda^2}\ \pi_1 \& \w^1 + \pi_2 \& \w^2 + \pi_3 \& \w^3,\\
-d\theta_2 &\equiv \pi_2 \& \w^1 + \pi_1 \& \w^2,\\
-d\theta_3 &\equiv \pi_3 \& \w^1 + \pi_1 \& \w^3
\end{aligned}
\mod \theta_0,\theta_1,\theta_2,\theta_3,
\end{equation}
where
\begin{align*}
\pi_1 &:= d\lambda -3\beta\lambda\w^3,\\
\pi_2 &:= d\beta + (\lambda^2-\beta^2)\ \w^3,\\
\pi_3 &:=\beta \w^3_2 -\dfrac{\beta(3\lambda^2+4c)}{\lambda}\ \w^1 -
(\beta^2+\lambda^2)\w^2.
\end{align*}
In order for the pullbacks of the right-hand sides in \eqref{I2struct} to vanish, $\pi_1,\pi_2$ and
$\pi_3$ must be multiples of $\w^1, \w^2, \w^3$ respectively---and moreover
the multiples must all be the same at each point.  In other words,
there must be a single function $\rho$ such that
$\pi_i =\rho \,\w^i$, $i=1\ldots 3$, at each point.

Therefore, we define the prolongation of $\I$ on $F_u \times \Sigma \times \R$, with
$\rho$ as new coordinate on the last factor, as the Pfaffian system generated
by $\theta_0,\ldots, \theta_3$ and the new 1-forms
$$
\theta_4 = \pi_1 - \rho\, \w^1, \quad
\theta_5 = \pi_2 -\rho\, \w^2, \quad
\theta_6 = \pi_3 -\rho\, \w^3.
$$
Now we compute
$$d\theta_5 \& \w^3 + d\theta_6 \& \w^2 \equiv 24\ \dfrac{c\beta^2}{\lambda}\ \w^1 \& \w^2 \& \w^3$$
modulo $\theta_0,\ldots, \theta_6$.  Since $\beta \ne 0$, this shows that
no integral submanifold of the prolongation can satisfy the independence condition.
Hence no hypersurfaces exist satisfying the hypothesis of the lemma.
\end{proof}


\begin{thebibliography}{10}
\bibitem{BCG3} R. Bryant, S.-S. Chern, R. Gardner, H. Goldschmidt, P. Griffiths,
{\it Exterior Differential Systems}, MSRI Publications, 1989.

\bibitem{cecilryan}
T.E. Cecil and P.J. Ryan,
{\sf Focal Sets and Real Hypersurfaces in Complex Projective Space},
Trans. Amer. Math. Soc. {\bf 269} (1982), 481--499.

\bibitem{cfb}
T.A. Ivey  and J.M. Landsberg,
{\it Cartan for Beginners: Differential geometry via moving frames and
exterior differential systems}, American Mathematical Society, 2003.
\bibitem{IveyRyan}
T.A. Ivey and P.J. Ryan,
{\sf Hopf Hypersurfaces of Small Hopf Principal Curvature in $\CH^2$}
Preprint, 2008.


\bibitem{KiSuh} U-H. Ki and Y.J. Suh,
{\sf On real hypersurfaces of a complex space form},
Math. J. Okayama Univ. {\bf 32} (1990), 207--221.


\bibitem{kimryan2}
H.S. Kim and P.J. Ryan,
{\sf A classification of pseudo-Einstein hypersurfaces
in $\CP^2$},
Differential Geom. Appl. {\bf 26} (2008),  106--112.


\bibitem{YMaeda} Y. Maeda,
{\sf On real hypersurfaces of a complex projective space},
J. Math. Soc. Japan {\bf 28} (1976), 529--540.

\bibitem{martins} J.K. Martins,
{\sf Hopf hypersurfaces in space forms}, Bull. Braz. Math. Soc.
(N.S.) {\bf 35} (2004), 453--472.

\bibitem{montiel} S. Montiel,
{\sf Real hypersurfaces of a complex hyperbolic space},
J. Math. Soc. Japan {\bf 37} (1985), 515--535.

\bibitem{nrsurvey}
R. Niebergall and P.J. Ryan, {\sf Real hypersurfaces in complex space forms}, pp. 233--205 in
{\it Tight and taut submanifolds} (ed. S.-S. Chern and
T.E. Cecil), MSRI Publications, 1997.


\bibitem{Ortega2006} M. Ortega, J.D. Perez, and F.G. Santos,
{\sf Non-existence of real hypersurfaces with parallel structure Jacobi operator in nonflat complex space forms},
Rocky Mountain J. Math. {\bf 36} (2006), 1603--1613.

\bibitem{Perez2005a} J.D. Perez and F.G. Santos,
{\sf On the Lie derivative of structure Jacobi operator of real hypersurfaces in complex projective space},
Publ. Math. Debrecen {\bf 66} (2005), 269--282.

\bibitem{Perez2005b} J.D. Perez, F.G. Santos, and Y.J. Suh,
{\sf Real hypersurfaces in complex projective space whose structure Jacobi operator is Lie $\xi$-parallel},
Differential Geom. Appl. {\bf 22} (2005), 181--188.



\bibitem{takagi1975}
R. Takagi, {\sf Real hypersurfaces in a complex projective space
with constant principal curvatures }, J. Math. Soc. Japan {\bf 27}
(1975), 45--53.

\end{thebibliography}
\end{document}